\documentclass[12pt,a4paper]{article}        

\usepackage{amssymb}
\usepackage{amsmath}
\usepackage{indentfirst}
\usepackage{t1enc}
\usepackage{array}
\usepackage{graphpap}
\usepackage[mathscr]{eucal}
\usepackage{enumerate}
\usepackage{tabularx}
\usepackage{graphicx}
\usepackage{amsthm}
\usepackage{float}
\usepackage{xurl}
\usepackage{color}
\usepackage[normalem]{ulem}
\usepackage[T1]{fontenc}
\usepackage[english]{babel}

\usepackage[hyperfootnotes=false, hidelinks]{hyperref}

\newtheorem{theorem}{Theorem}
\newtheorem{definition}{Definition}
\newtheorem{lemma}{Lemma}
\newtheorem{problem}{Problem}
\newtheorem{corollary}{Corollary}

\newtheorem{prop}{Proposition}

\newcommand{\mc}{\mathcal}

\newcommand{\ab}[1]{\left\vert{#1}\right\vert}
\newcommand{\zj}[1]{\left({#1}\right)}

\renewcommand{\baselinestretch}{1.38}

\newcommand{\lb}[1]{\label{#1}}
\allowdisplaybreaks[1] 

\normalsize
\title{Translation-invariant equations in $\mathbb{F}_p^n$ and groups}

\author{Katalin Gyarmati}
\date{}

\begin{document}

\maketitle


\begin{abstract}
In this paper, we explore how large a subset of a finite group can be without containing non-trivial solutions to linear equations of the form $aX+bY=cZ+dU$ (where $a+b=c+d$). These equations naturally generalize the classic concept of Sidon sets. Using Tao's slice rank method, we first establish strict upper bounds for the size of such solution-free sets in the vector space $\mathbb{F}_{p}^{n}$. Next, we turn to cyclic groups $\mathbb{Z}_{m}$ and construct surprisingly large sets that have no solutions for the asymmetric equation $X+5Y=3U+3Z$. These constructions achieve a logarithmic density of roughly $0.5283$, breaking the expected $0.5$ barrier for any sufficiently large modulus $m$. Finally, we show that this high-density behavior extends to a wider family of equations whose coefficients follow simple rules modulo $8$.
\end{abstract}

\section{Introduction}

The study of solution-free sets for linear equations is a central and active research area in additive combinatorics. The roots of this intriguing subject trace back to Roth's seminal 1953 theorem \cite{roth1953}, which states that any subset of natural numbers with positive upper density contains a three-term arithmetic progression (which essentially guarantees the presence of non-trivial solutions to the equation $X+Y=2Z$). Equally foundational is the study of Sidon sets, defined by the absence of non-trivial solutions to the symmetric equation $X+Y=Z+U$. The equations studied in this paper can be naturally viewed as broad, asymmetric generalizations of this classic Sidon property.

The behavior of a much broader class of such linear equations was studied by Imre Ruzsa in a 1993 paper \cite{ruzsa1993}. In this work, among other variations, he systematically investigated equations of the form $aX + bY = cU + dV$, where $a, b, c, d > 0$ are integers, and $a+b = c+d$. Let $r(N)$ denote the maximum cardinality of a subset $\mc A \subset \{1, \dots, N\}$ in which a given equation $aX+bY=cZ+dU$ has no non-trivial solutions. According to Theorem 7.3 in \cite{ruzsa1993}, if the product of the coefficients $abcd$ is not a square, then $r(N) \gg \sqrt{N}$, while if it is a square, then $r(N) \gg \sqrt{N} e^{-\beta \sqrt{\log N}}$ holds for an appropriate constant $\beta > 0$. 

In the proof of this theorem, Ruzsa relies heavily on modular arithmetic. In connection with this (based on personal communication), he raised the following two problems:

\begin{problem}\lb{prob:ruzsa1} 
  Let $2\le m \in \mathbb{Z}$, and consider the residue class ring $\mathbb{Z}_m$. Let $a,b,c,d$ be fixed elements of $\mathbb Z_m$, satisfying $a+b=c+d$. Let $s(m)$ denote the maximum size of a subset $\mc A \subseteq \mathbb{Z}_m$ in which the equation $aX + bY = cZ + dU$ (where $X,Y,Z,U\in \mc A$) has only the trivial solution $X=Y=Z=U$. Is it true that there exists a constant $\varepsilon > 0$ such that $s(m) \le m^{1-\varepsilon}$ holds for all sufficiently large $m$?
\end{problem}

\begin{problem}\lb{prob:ruzsa2}
Do there exist positive integers $a,b,c,d$ and a modulus $m$ for which, using the previous notation, $s(m) > \sqrt{m}$ holds?
\end{problem}

From the proof of Ruzsa's Theorem 7.3 \cite{ruzsa1993}, we know that $s(m)\gg \sqrt{m}$ if $abcd$ is not a square, and $s(m) \gg \sqrt{m} e^{-\beta \sqrt{\log m}}$ if $abcd$ is a square; where the implied constants also depend on $a,b,c,d$. In this paper, we answer Problem \ref{prob:ruzsa2} by proving that for the asymmetric equation $X+5Y=3U+3Z$, the maximum size of a solution-free set can actually break the expected $m^{0.5}$ barrier, reaching roughly $m^{0.5283}$ (the exact result is given in Theorem \ref{thm03}). We also show that this method naturally extends to a broader family of linear equations $aX+bY=cZ+dU$ with specific modulo 8 coefficients, yielding the same high density (see Theorem \ref{thm04}). While solving these problems completely in $\mathbb{Z}_m$ is still open, switching from $\mathbb{Z}_m$ to other finite structures like the vector space $\mathbb{F}_p^n$) yields sharp results. Recent breakthroughs by Croot--Lev--Pach and Ellenberg--Gijswijt \cite{croot2017, ellenberg2017} have revolutionized how we handle these combinatorial problems. One of the cleanest ways to use these methods is Tao's slice rank lemma \cite{tao2016symmetric}, which yields the following explicit upper bound:

\begin{theorem}\label{thm01}
Let $p$ be a prime, and let $a, b, c, d \in \mathbb{F}_p$ be elements such that $a+b=c+d$. If a subset $\mathcal{A} \subseteq \mathbb{F}_p^n$ is such that the equation $aX + bY = cZ + dU$ has no non-trivial solutions (i.e., any solution satisfies $X=Y=Z=U$), then for the cardinality of $\mathcal{A}$ we have
\[
|\mathcal{A}| \le 4\left(\frac{e(p+3)}{4}\right)^n.
\]
\end{theorem}

If we apply Lemma \ref{lem01} (see later) to the prime $p=5$ and the group $G=\mathbb{F}_{5}^n$ (where the bound can be numerically minimized), we obtain the following asymptotic result:

\begin{corollary}\lb{cor01a}
  For every number $N \in \mathbb{N}$ there exists a finite abelian group $G$ with cardinality $|G| \ge N$ of exponent $5$, having the following property:
Suppose that $a, b, c, d \in \mathbb{Z}$ are fixed coefficients with $a+b=c+d$. Suppose that the equation
$$aX + bY = cZ + dU$$
has only the trivial solution in a set $\mc A \subseteq G$. Then the cardinality of the set satisfies
$$|\mc A| \le |G|^{0.83509}.$$
\end{corollary}

In Section 4, we return to the setting where $G = \mathbb{Z}_m$. We employ extensive computational searches to systematically study the behavior of extremal sets for $m \le 122$ with respect to the equation 
\[
2X + 2Y \equiv U + 3Z \pmod{m}.
\]
Following this empirical analysis, we turn our attention to another strongly asymmetric equation, $X + 5Y \equiv 3U + 3Z$, and provide explicit asymptotic constructions. These constructions demonstrate rigorously that for certain moduli, the maximum size of solution-free sets can exceed the anticipated bounds:

\begin{prop}\label{prop01}
Let the modulus be of the form $m = 4^n$ for a positive integer $n$. Then there exists a subset $\mathcal{A} \subseteq \mathbb{Z}_m$ of cardinality 
\[
|\mathcal{A}| = 2^n = m^{1/2},
\]
such that the equation $X + 5Y \equiv 3U + 3Z \pmod{m}$ has no non-trivial solutions.
\end{prop}

As an interesting fact, we mention that if, instead of the group $\mathbb Z_{4^n}$, we try to construct a subset $\mathcal{A}$ in another group of order $m$ such that the equation $X+5Y=3U+3Z$ has only trivial solutions in it, then the cardinality of $\mathcal{A}$ can reach $m^{1/2+\varepsilon}$. More precisely:

\begin{theorem}\lb{thm02a}
  Let $m=8^{3n}$ and the group $G$ of order $m$ be $G := \mathbb Z_8^{3n}$. Then there exists a set $\mathcal{A} \subseteq G$ in which the equation $X+5Y=3U+3Z$ has only trivial solutions, and
\begin{equation}  
  |\mathcal{A}| \ge 0.2 \frac{3^{3n}}{n} \gg \frac{m^{\log 3/\log 8}}{\log m}
  \approx \frac{m^{0.5283\dots}}{\log m}.\lb{r03}
\end{equation}  
\end{theorem}

Theorem \ref{thm02a} proves that a large solution-free set can be constructed within $\mathbb{Z}_8^{3n}$. It is a natural question whether this high density just is due to this specific group, or can we get a similar bound in the standard cyclic group $\mathbb{Z}_m$ for arbitrarily large moduli. The following theorem provides an affirmative answer, showing that the asymptotic behavior can be extended to general $m$, keeping the exponent $\log 3/\log 8$ minus an arbitrarily small $\varepsilon > 0$.

\begin{theorem}\lb{thm03}
For every constant $\varepsilon > 0$, there exists an $m_0$ such that for all $m > m_0$, there exists a set $\mathcal{A} \subseteq \mathbb{Z}_m$ in which the equation
\[ X + 5Y = 3U + 3Z \]
has only the trivial solution ($X = Y = U = Z$), and
\[ \ab{\mc{A}} \ge m^{\log 3/ \log 8 - \varepsilon}\approx m^{0.5283\dots} . \]
\end{theorem}

The method used to prove Theorem \ref{thm03} naturally extends to a wider class of linear equations. This is because our base-8 construction depends only on the modulo 8 remainders of the coefficients, not on their exact values. Since the proof of the following, more general theorem is very similar to that of Theorem \ref{thm03}, we will only sketch it in the final section.

\begin{theorem}\lb{thm04}
For every constant $\varepsilon > 0$, and for any integers $a, b, c, d$ satisfying $a \equiv 1 \pmod 8$, $b \equiv 5 \pmod 8$, $c \equiv d \equiv 3 \pmod 8$, and $a+b=c+d$, there exists an $m_0$ such that for all $m > m_0$, there exists a set $\mathcal{A} \subseteq \mathbb{Z}_m$ in which the equation
\[ aX + bY = cZ + dU \]
has only the trivial solution ($X = Y = U = Z$), and
\[ \ab{\mc{A}} \ge m^{\log 3/ \log 8 - \varepsilon}\approx m^{0.5283\dots} . \]
\end{theorem}

\vspace{0.5cm}
\noindent \textbf{Note added.} Shortly before the submission of this manuscript, a preprint by Hametner and Solymosi (arXiv:2609.13972) appeared, which independently investigates large solution-free sets for four-variable asymmetric linear equations. While their work also breaks the expected $0.5$ exponent barrier, their focus and methodology are entirely distinct from ours. They employ the method of combinatorial degenerations and finite certificates to study the equation $3x+y=2z+2w$, whereas the present paper utilizes Tao's slice rank method and explicit base-8 constructions to establish bounds for the equation $X+5Y=3U+3Z$.
\vspace{0.5cm}

\section{Proof of Theorem \ref{thm01}.}

Tensors provide a natural algebraic framework for studying multivariate polynomials. In a combinatorial context, a tensor can be viewed as a multidimensional matrix or a multivariate function.

\begin{definition}[Tensor]
Let $A$ be a finite set, $k \ge 2$ an integer, and $\mathbb{F}_p$ a finite field. By a $k$-ary tensor we mean a function $T: A^k \to \mathbb{F}_p$ defined on the $k$-fold Cartesian product of the set. 
\end{definition}

The concept of matrix rank can be generalized to higher-dimensional tensors in several ways. For us, the most important is the slice rank, which excellently reflects combinatorial properties.

\begin{definition}[Slice Rank]
The slice rank (denoted by $\operatorname{s-rank}(T)$) of a tensor $T: A^k \to \mathbb{F}_p$ is the smallest integer $R$ for which $T$ can be written as a sum of terms of the following form:
\begin{equation*}
T(x_1, \ldots, x_k) = \sum_{i=1}^R f_i(x_{j_i}) \cdot g_i(x_1, \ldots, \widehat{x}_{j_i}, \ldots, x_k)
\end{equation*}
where $j_i \in \{1, \ldots, k\}$ is an index, $f_i: A \to \mathbb{F}_p$ is a single-variable function, and $g_i: A^{k-1} \to \mathbb{F}_p$ is an arbitrary function depending on the remaining $k-1$ variables (the notation $\widehat{x}_{j_i}$ means that the variable $x_{j_i}$ is omitted from the arguments).
\end{definition}

The efficacy of the slice rank method relies on the fact that the slice rank of a diagonal tensor is exactly equal to the size of its support set.

\begin{lemma}[Tao's slice rank lemma \cite{tao2016symmetric}, 2016]
Let $A$ be a finite set, and $T: A^k \to \mathbb{F}_p$ a diagonal tensor for which $T(x_1, \ldots, x_k) \neq 0$ holds if and only if $x_1 = x_2 = \ldots = x_k$. Then the slice rank of the tensor is exactly equal to the cardinality of the set:
\begin{equation*}
\operatorname{s-rank}(T) = |A|
\end{equation*}
\end{lemma}
Related to this lemma, see also \cite{tao2016symmetric, tao2016notes, blasiak2017capsets}.

Suppose that in the set $\mc A \subseteq \mathbb{F}_p^n$ the equation $aX + bY - cZ - dU = \mathbf{0}$ has only trivial ($X=Y=Z=U$) solutions. Define the following $4$-variable indicator polynomial:
\begin{equation*}
F(X,Y,Z,U) = \prod_{i=1}^n \left( 1 - (ax_i + by_i - cz_i - du_i)^{p-1} \right),
\end{equation*}
where $X=(x_1,x_2,\dots,x_n)$,  $Y=(y_1,y_2,\dots,y_n)$,  $U=(u_1,u_2,\dots,u_n)$ and  $Z=(z_1,z_2,\dots,z_n)$. 
Restricted to the set $\mc A \times \mc A \times \mc A \times \mc A$, this function $F$ is a diagonal tensor, so by Tao's theorem $\operatorname{s-rank}(F) = |\mc A|$.

\bigskip
On the other hand, $F$ is a polynomial whose total degree is at most $n(p-1)$. Each monomial of the expanded product is of the form $f_1(X)f_2(Y)f_3(Z)f_4(U)$, where $f_1, f_2, f_3, f_4$ are monomials themselves. Since the sum of the degrees of these four factors is at most $n(p-1)$, by the pigeonhole principle, the degree of at least one of the four factors ($f_1(X),f_2(Y),f_3(Z),f_4(U)$) cannot exceed a quarter of the total degree, i.e., the value $\frac{p-1}{4}n$.

\bigskip
Accordingly, the polynomial $F$ can be decomposed into a sum of four polynomials in each of which the degree of one selected factor ($f_1(X),f_2(Y),f_3(Z)$ or $f_4(U)$) is at most $\frac{p-1}{4}n$. By the definition of the slice rank, this implies that $\operatorname{s-rank}(F) \le 4M$, where $M$ is the number of $n$-variable monomials in which the total degree does not exceed $\frac{p-1}{4}n$, and the maximum degree of each coordinate is $p-1$. Comparing the two rank relations yields the upper bound:
\begin{equation*}
  |\mc{A}| \le 4M. %\lb{fofo}
\end{equation*}

\bigskip
Our goal is to give an upper bound for the number $M$ of $n$-variable monomials in which the exponent of each variable is at most $p-1$, and the sum of the exponents (the total degree) does not exceed $\frac{p-1}{4}n$. For this, we call upon an elementary combinatorial and analytical tool: generating functions.

\bigskip
Consider the following single-variable polynomial:
\begin{equation*}
P(x) = (1 + x + x^2 + \dots + x^{p-1})^n.
\end{equation*}
If we expand this expression, the coefficient of the term $x^k$ -- let us denote it by $c_k$ -- gives exactly the number of ways we can choose $n$ integer exponents, each at most $p-1$, such that their sum is exactly $k$. 

The number of monomials $M$ we are interested in is the sum of these coefficients from $k=0$ up to $K = \lfloor \frac{p-1}{4}n \rfloor$:
\begin{equation*}
M = \sum_{k=0}^K c_k.
\end{equation*}

To estimate this sum, we apply a simple algebraic trick. Let $x$ be an arbitrary real number in the interval $(0, 1)$. Since $x < 1$, for any exponent $k$ such that $k \le K$, the inequality $k - K \le 0$ holds, which implies that:
\begin{equation*}
x^{k-K} \ge 1.
\end{equation*}

Using this fact, we can multiply each coefficient $c_k$ by the value $x^{k-K}$, which is greater than or equal to $1$, so the sum can only increase. We further increase the sum if we also add the terms with indices greater than $K$ (since the coefficients are non-negative and $x > 0$):
\begin{equation*}
M = \sum_{k=0}^K c_k \le \sum_{k=0}^K c_k x^{k-K} \le \sum_{k=0}^{n(p-1)} c_k x^{k-K}.
\end{equation*}

We can factor out the term $x^{-K}$ from the expression on the right side, and the remaining sum is exactly the original form of the polynomial $P(x)$:
\begin{equation*}
M \le x^{-K} \sum_{k=0}^{n(p-1)} c_k x^k = x^{-K} P(x).
\end{equation*}

Substituting $K = \frac{p-1}{4}n$ and the definition of $P(x)$:
\begin{align*}
  M &\le x^{-(p-1)n/4} (1 + x + x^2 + \dots + x^{p-1})^n\\
  &= \left( x^{-(p-1)/4} (1 + x + x^2 + \dots + x^{p-1}) \right)^n.
\end{align*}

Thus, we have established the following lemma:
\begin{lemma}\lb{lem01}
Using the notation and definitions of Theorem \ref{thm01}, the following inequality holds for the set $\mc A$:
\begin{align*}
|\mc A| \le 4\zj{\min_{0<x<1} x^{-(p-1)/4} (1+x+\dots+x^{p-1})}^n.%\lb{korhoz}
\end{align*}
\end{lemma}

Let the base of the exponentiation be the function $f(x)$:
\begin{equation*}
f(x) = x^{-(p-1)/4} (1 + x + x^2 + \dots + x^{p-1}).
\end{equation*}
Our goal is to show that $f(x) < \frac{e(p+3)}{4}$ holds for some $0 < x < 1$, which will complete the proof of the theorem.

\bigskip
Rewriting the function using the formula for the sum of a geometric progression:
\begin{equation}
  f(x) = x^{-(p-1)/4} \frac{1-x^p}{1-x}<\frac{1}{x^{(p-1)/4} (1-x)}.
\lb{11}
\end{equation}  

Our goal is to make this upper bound as small as possible, which requires maximizing the denominator. For ease of writing, let us introduce the notation $k = \frac{p-1}{4}$. We are looking for the maximum of the expression $g(x) = x^k (1-x)$ on the interval $(0,1)$.

\bigskip
Differentiating using the product rule:
\begin{align*}
  g'(x) &= k x^{k-1} (1-x) - x^k\\
        &= x^{k-1} (k - x(k+1)).
\end{align*}        
The maximum is where the derivative is zero, that is:
\begin{align*}
k - x(k+1) &= 0\\
  x &= \frac{k}{k+1}.
\end{align*}   
Substituting back the value $k = \frac{p-1}{4}$ to obtain that the function $\frac{1}{x^{\frac{p-1}{4}} (1-x)}$ takes its minimum at the point
$$x = \frac{\frac{p-1}{4}}{\frac{p-1}{4} + 1} = \frac{p-1}{p+3}.$$
Clearly, $0<x<1$.

\bigskip
Inserting this optimal value $x = \frac{p-1}{p+3}$ into our upper bound \eqref{11}, we obtain:
\begin{align*}
f\left(\frac{p-1}{p+3}\right) &< \frac{p+3}{4} \left(\frac{p+3}{p-1}\right)^k\\
                              &=\frac{p+3}{4} \left(1 + \frac{1}{k}\right)^k.
\end{align*}                              
Here
$$\left(1 + \frac{1}{k}\right)^k < e.$$
By this, we obtain the final, closed-form, strict upper bound:
$$f\left(\frac{p-1}{p+3}\right) < \frac{e}{4} (p+3),$$ which completes the
proof of the theorem. \qed

\section{Proof of Corollary \ref{cor01a}.}

To prove the statement of the Corollary, we apply Lemma \ref{lem01} for the prime number $p = 5$. Fix an arbitrary number $N \in \mathbb{N}$. Let the base group be of the form $G = \mathbb{F}_{5}^n$ (which, considered as an additive group, is isomorphic to the group $\mathbb{Z}_{5}^n$), where the dimension $n$ is chosen large enough so that $|G| = 5^n \ge N$ is satisfied.

\bigskip
According to Lemma \ref{lem01}, the following upper bound holds for the cardinality of the solution-free set $\mc A \subseteq \mathbb{F}_{5}^n$:
$$|\mc A| \le 4\zj{\min_{0<x<1} x^{-(5-1)/4} (1+x+\dots+x^{5-1})}^n.$$
This minimum is not larger than the value of the function
$x^{-(5-1)/4} (1+x+\dots+x^{5-1})$ at the point $0.56774$, which
is less than $3.83444$, thus
\begin{align*}
  |\mc A| &< 4 \cdot \left(3.83444\right)^n
  = 4 \cdot \left(5^n\right)^{\log (3.83444/\log (5)}\\
  &< 4 \cdot |G|^{0.835089}< |G|^{0.83509},
\end{align*}
provided that the cardinality $|G|$ of the group is large enough.
This completes the proof of the corollary. \qed

\section{Computer search with the condition $0, 1 \in \mc{A}$}

In this section, using computers, we searched for examples for $m\le 122$ of maximum size sets $\mc A\subset \mathbb Z_m$ that contain the elements $0$ and $1$, and in which the equation $2X+2Y\equiv U+3Z\pmod{m}$ has only trivial solutions. 

\bigskip
The columns of the table provide the modulus $m$, the maximum set size ($|\mc{A}|$), the logarithmic density parameter ($\frac{\log |\mc{A}|}{\log m}$), and a specific extremal set $\mc{A}$, respectively. For the cases $m=2, 3, 4$, no such set exists (i.e., there is no subset containing both $0$ and $1$ that avoids non-trivial solutions); thus, no values are provided for these rows. To save space, the table presents the results divided into three columns, and the elements of the sets are listed without spaces.

\begin{table}[H]
\centering
\renewcommand{\baselinestretch}{1.1}
\setlength{\tabcolsep}{3pt}
\scriptsize
\begin{tabular}{|c c c c | c c c c | c c c c|}
\hline
$m$ & $|\mc{A}|$ & $\frac{\log |\mc{A}|}{\log m}$ & $\mc{A}$ & $m$ & $|\mc{A}|$ & $\frac{\log |\mc{A}|}{\log m}$ & $\mc{A}$ & $m$ & $|\mc{A}|$ & $\frac{\log |\mc{A}|}{\log m}$ & $\mc{A}$ \\
\hline
2 & - & - & - & 43 & 4 & 0.3686 & \{0,1,8,9\} & 84 & 6 & 0.4044 & \{0,1,5,30,34,35\} \\
3 & - & - & - & 44 & 4 & 0.3663 & \{0,1,7,38\} & 85 & 6 & 0.4033 & \{0,1,11,12,27,70\} \\
4 & - & - & - & 45 & 5 & 0.4228 & \{0,1,11,18,20\} & 86 & 6 & 0.4023 & \{0,1,6,13,67,79\} \\
5 & 2 & 0.4307 & \{0,1\} & 46 & 4 & 0.3621 & \{0,1,7,8\} & 87 & 6 & 0.4012 & \{0,1,5,6,26,67\} \\
6 & 2 & 0.3869 & \{0,1\} & 47 & 4 & 0.3601 & \{0,1,5,11\} & 88 & 6 & 0.4002 & \{0,1,5,11,34,42\} \\
7 & 2 & 0.3562 & \{0,1\} & 48 & 4 & 0.3581 & \{0,1,6,7\} & 89 & 6 & 0.3992 & \{0,1,7,18,22,62\} \\
9 & 2 & 0.3155 & \{0,1\} & 50 & 5 & 0.4114 & \{0,1,9,11,29\} & 91 & 7 & 0.4314 & \{0,1,21,27,73,78,80\} \\
10 & 2 & 0.3010 & \{0,1\} & 51 & 4 & 0.3526 & \{0,1,5,6\} & 92 & 6 & 0.3962 & \{0,1,11,13,18,84\} \\
11 & 2 & 0.2891 & \{0,1\} & 52 & 5 & 0.4073 & \{0,1,6,15,44\} & 93 & 6 & 0.3953 & \{0,1,8,9,19,83\} \\
12 & 2 & 0.2789 & \{0,1\} & 53 & 5 & 0.4054 & \{0,1,7,8,17\} & 94 & 6 & 0.3944 & \{0,1,6,7,15,86\} \\
13 & 2 & 0.2699 & \{0,1\} & 54 & 5 & 0.4035 & \{0,1,5,11,42\} & 95 & 7 & 0.4273 & \{0,1,25,34,53,57,82\} \\
14 & 2 & 0.2627 & \{0,1\} & 55 & 5 & 0.4016 & \{0,1,6,7,15\} & 96 & 6 & 0.3926 & \{0,1,9,10,16,90\} \\
15 & 2 & 0.2559 & \{0,1\} & 56 & 5 & 0.3998 & \{0,1,5,26,50\} & 97 & 6 & 0.3917 & \{0,1,5,6,23,62\} \\
16 & 2 & 0.2500 & \{0,1\} & 57 & 4 & 0.3429 & \{0,1,7,8\} & 98 & 6 & 0.3908 & \{0,1,11,14,18,90\} \\
17 & 2 & 0.2447 & \{0,1\} & 58 & 5 & 0.3964 & \{0,1,21,24,52\} & 99 & 6 & 0.3899 & \{0,1,8,14,65,69\} \\
18 & 2 & 0.2398 & \{0,1\} & 59 & 5 & 0.3947 & \{0,1,9,11,53\} & 100 & 6 & 0.3891 & \{0,1,5,14,18,19\} \\
19 & 3 & 0.3731 & \{0,1,8\} & 60 & 5 & 0.3930 & \{0,1,5,11,54\} & 101 & 7 & 0.4216 & \{0,1,12,27,30,43,49\} \\
20 & 2 & 0.2314 & \{0,1\} & 61 & 5 & 0.3915 & \{0,1,6,10,50\} & 102 & 6 & 0.3874 & \{0,1,10,11,18,95\} \\
21 & 3 & 0.3609 & \{0,1,9\} & 62 & 5 & 0.3899 & \{0,1,5,14,44\} & 103 & 7 & 0.4199 & \{0,1,5,21,24,46,56\} \\
22 & 3 & 0.3554 & \{0,1,9\} & 63 & 6 & 0.4325 & \{0,1,7,12,15,40\} & 104 & 8 & 0.4477 & \{0,1,5,14,18,19,39,84\} \\
23 & 3 & 0.3504 & \{0,1,7\} & 64 & 5 & 0.3869 & \{0,1,5,24,26\} & 105 & 7 & 0.4181 & \{0,1,6,7,15,67,83\} \\
24 & 3 & 0.3456 & \{0,1,10\} & 65 & 6 & 0.4292 & \{0,1,8,10,56,58\} & 106 & 6 & 0.3842 & \{0,1,5,14,19,75\} \\
25 & 4 & 0.4307 & \{0,1,10,11\} & 66 & 5 & 0.3841 & \{0,1,6,7,58\} & 107 & 7 & 0.4164 & \{0,1,11,28,29,61,66\} \\
26 & 4 & 0.4255 & \{0,1,5,6\} & 67 & 5 & 0.3828 & \{0,1,6,13,15\} & 108 & 6 & 0.3827 & \{0,1,7,18,23,32\} \\
27 & 4 & 0.4206 & \{0,1,5,6\} & 68 & 5 & 0.3814 & \{0,1,6,7,15\} & 109 & 7 & 0.4148 & \{0,1,5,6,23,47,85\} \\
28 & 4 & 0.4160 & \{0,1,11,12\} & 69 & 6 & 0.4232 & \{0,1,25,26,42,53\} & 110 & 7 & 0.4140 & \{0,1,6,20,23,25,71\} \\
29 & 4 & 0.4117 & \{0,1,6,7\} & 70 & 6 & 0.4217 & \{0,1,10,12,33,51\} & 111 & 7 & 0.4132 & \{0,1,7,32,39,51,67\} \\
30 & 4 & 0.4075 & \{0,1,5,26\} & 71 & 6 & 0.4203 & \{0,1,6,27,32,33\} & 112 & 7 & 0.4124 & \{0,1,6,33,39,43,54\} \\
31 & 4 & 0.4037 & \{0,1,14,18\} & 72 & 6 & 0.4189 & \{0,1,6,7,15,64\} & 113 & 7 & 0.4116 & \{0,1,5,6,35,92,96\} \\
32 & 4 & 0.4000 & \{0,1,5,6\} & 73 & 6 & 0.4176 & \{0,1,11,48,58,59\} & 114 & 7 & 0.4109 & \{0,1,8,10,23,26,87\} \\
33 & 4 & 0.3965 & \{0,1,5,6\} & 74 & 6 & 0.4163 & \{0,1,21,22,45,51\} & 115 & 8 & 0.4382 & \{0,1,8,20,23,66,100,106\} \\
34 & 4 & 0.3931 & \{0,1,7,8\} & 75 & 6 & 0.4150 & \{0,1,6,7,15,67\} & 116 & 7 & 0.4093 & \{0,1,5,6,28,47,49\} \\
35 & 4 & 0.3899 & \{0,1,6,7\} & 76 & 6 & 0.4137 & \{0,1,15,41,55,56\} & 117 & 7 & 0.4086 & \{0,1,6,7,15,50,95\} \\
36 & 4 & 0.3869 & \{0,1,6,7\} & 77 & 6 & 0.4125 & \{0,1,9,10,16,71\} & 118 & 7 & 0.4079 & \{0,1,24,26,29,62,99\} \\
37 & 4 & 0.3839 & \{0,1,8,9\} & 78 & 6 & 0.4113 & \{0,1,7,12,47,69\} & 119 & 8 & 0.4351 & \{0,1,6,25,42,92,102,104\} \\
38 & 4 & 0.3811 & \{0,1,5,6\} & 79 & 6 & 0.4101 & \{0,1,6,7,17,69\} & 120 & 8 & 0.4344 & \{0,1,5,14,22,75,96,111\} \\
39 & 4 & 0.3784 & \{0,1,7,8\} & 80 & 6 & 0.4089 & \{0,1,8,10,23,71\} & 121 & 7 & 0.4058 & \{0,1,5,16,18,37,97\} \\
40 & 4 & 0.3758 & \{0,1,5,16\} & 81 & 6 & 0.4077 & \{0,1,5,6,23,64\} & 122 & 7 & 0.4051 & \{0,1,8,14,23,58,117\} \\
41 & 4 & 0.3733 & \{0,1,6,10\} & 82 & 6 & 0.4066 & \{0,1,6,7,15,74\} & & & & \\
42 & 4 & 0.3709 & \{0,1,5,6\} & 83 & 6 & 0.4055 & \{0,1,7,12,15,50\} & & & & \\
\hline
\end{tabular}
\end{table}

Based on the computational data presented in the table, we can draw the following conclusions regarding the extremal sets for $m \le 122$:

\begin{itemize}
    \item[] For the tested values, the logarithmic density $\frac{\log |\mc{A}|}{\log m}$ attains its maximum at $m = 104$, taking the value of $0.4477$.

    \item[] Across the whole table, this density remains below $0.5$. This might seem to contrast with Ruzsa's Theorem 7.3 in \cite{ruzsa1993} (where $s(m)\ge c_0m^{0.5}$ for some constant $c_0<1$). However, as we will show later, for another strongly asymmetric equation, $X+5Y\equiv 3U+3Z\pmod{m}$, the density can actually exceed $0.5$ for sufficiently large moduli.

    \item[] The maximum size of the solution-free set, $|\mc{A}|$, grows very slowly as the modulus $m$ increases. For $m \le 122$, the largest size we found is just $|\mc{A}| = 8$, which appears at $m = 104, 115, 119$, and $120$.

    \item[] For very small moduli ($m = 2, 3, 4$), there are no valid sets containing both $0$ and $1$, meaning that the algebraic structure is initially too restrictive to avoid non-trivial solutions.
\end{itemize}

\section{Proof of Proposition \ref{prop01}.}

To prove this, we construct the set $\mc{A}$ as follows: consider the integers $0 \le X < 4^n$ whose base-four representation contains exclusively the digits $0$ and $1$. That is, $X \in \mc{A}$ holds if and only if
$$X = \sum_{i=0}^{n-1} x_i 4^i, \quad \text{where } x_i \in \{0, 1\} \text{ for all } 0 \le i < n.$$
Then
$$|\mc{A}| = 2^n = (4^n)^{1/2} = m^{0.5}.$$
We need to show that the equation $X + 5Y \equiv 3U + 3Z \pmod{4^n}$ has only the trivial solution $X = Y = U = Z$ on the set $\mc{A}$. Suppose that $X, Y, U, Z \in \mc{A}$ satisfy the congruence. Let the base-four representations of the four numbers be $X = \sum x_i 4^i$, $Y = \sum y_i 4^i$, $U = \sum u_i 4^i$, and $Z = \sum z_i 4^i$, respectively. Then
$$\sum_{i=0}^{n-1} x_i 4^{i} + 5 \sum_{i=0}^{n-1} y_i 4^{i} \equiv 3 \sum_{i=0}^{n-1} u_i 4^{i} + 3 \sum_{i=0}^{n-1} z_i 4^{i} \pmod{4^{n}}.$$
Reducing the congruence modulo 4, we obtain:
\begin{align*}
x_0 + 5y_0 &\equiv 3u_0 + 3z_0 \pmod 4\\
  x_0 + y_0 + u_0 + z_0  &\equiv 0 \pmod 4.
\end{align*}                         
Since by construction $x_0, y_0, u_0, z_0 \in \{0, 1\}$, the sum of the four digits is an integer that falls into the closed interval $[0, 4]$. In this interval, only $0$ and $4$ are divisible by four. If the sum is $0$, then $x_0 = y_0 = u_0 = z_0 = 0$. If the sum is $4$, then $x_0 = y_0 = u_0 = z_0 = 1$.

Subtracting the common value $x_0+5y_0=3u_0+3z_0$ from the original congruence, and then dividing by $4$, we deduce the congruence for the following:
$$\sum_{i=1}^{n-1} x_i 4^{i-1} + 5 \sum_{i=1}^{n-1} y_i 4^{i-1} \equiv 3 \sum_{i=1}^{n-1} u_i 4^{i-1} + 3 \sum_{i=1}^{n-1} z_i 4^{i-1} \pmod{4^{n-1}}.$$

We can repeat the modulo $4$ examination for this congruence in a completely analogous way for the digits $x_1, y_1, u_1, z_1$. Proceeding by mathematical induction, it follows that $x_i = y_i = u_i = z_i$ holds for every index $i \in \{0, 1, \dots, n-1\}$. Since all digits of the four numbers are identical, we conclude that $X = Y = U = Z$. \qed

\section{Proof of Theorem \ref{thm02a}.}
Let $\mathcal{A} \subseteq \mathbb{Z}_8^{3n}$ be the set of all vectors containing exactly $n$ zeros, $n$ ones, and $n$ twos as coordinates. Then we have
\[
|\mathcal{A}| = \frac{(3n)!}{(n!)^3}.
\]
Based on the paper of Robbins
\cite{Robbins}, for every integer $n \ge 1$ we have:
$$\sqrt{2\pi n}\left(\frac{n}{e}\right)^n \cdot e^{\frac{1}{12n+1}} < n! < \sqrt{2\pi n}\left(\frac{n}{e}\right)^n \cdot e^{\frac{1}{12n}}.$$
Since $1 < e^{\frac{1}{12n+1}}$ and $\sqrt{2\pi} e^{1/12} < e$, thus
$$\sqrt{2\pi n}\left(\frac{n}{e}\right)^n < n! < e\sqrt{n}\left(\frac{n}{e}\right)^n.$$
Then
\begin{equation}\lb{r401} |\mathcal{A}| = \frac{(3n)!}{n!n!n!} > \frac{\sqrt{2\pi(3n)}\left(\frac{3n}{e}\right)^{3n}}{e^3 (\sqrt{n})^3 \left(\frac{n}{e}\right)^{3n}} = \frac{\sqrt{6\pi}}{e^3} \cdot \frac{3^{3n}}{n} > 0.2 \cdot \frac{3^{3n}}{n}. \end{equation}

With this, we have proved \eqref{r03}. We only need to prove that the equation $X+5Y=3U+3Z$ has only trivial solutions in $\mathcal{A}$.

For this, assume that $X+5Y=3U+3Z$. Let $X=(x_1, x_2, \dots, x_{3n})$, $Y=(y_1, y_2, \dots, y_{3n})$, $Z=(z_1, z_2, \dots, z_{3n})$, and $U=(u_1, u_2, \dots, u_{3n})$. From the assumption, for every coordinate:
$$x_i + 5y_i \equiv 3u_i + 3z_i \pmod 8.$$
Since $x_i, y_i, u_i, z_i \in \{0, 1, 2\}$, this can only be satisfied if
\begin{align*}
  (x_i, y_i, z_i, u_i) \in \{ &(0,0,0,0), (1,0,1,2), (1,0,2,1), (1,1,0,2), (1,1,2,0),\\
  &(1,1,1,1), (1,2,0,1), (1,2,1,0), (2,2,2,2)\}.
\end{align*}  
In the first column of the table below, we listed these 9 quadruples, and in the second column, we denoted with a letter of the English alphabet how many coordinates $1 \le i \le 3n$ there are for which $x_i, y_i, u_i, z_i$ match the given quadruple.

\begin{center}
\begin{tabular}{cccc|c}
$x_i$, & $y_i$,& $u_i$,& $z_i$ & how many $i$? \\
\hline
$0,$ & $0,$ & $0,$ & $0$ & $a$ \\
$1,$ & $0,$ & $1,$ & $2$ & $b$ \\
$1,$ & $0,$ & $2,$ & $1$ & $c$ \\
$1,$ & $1,$ & $0,$ & $2$ & $d$ \\
$1,$ & $1,$ & $2,$ & $0$ & $e$ \\
$1,$ & $1,$ & $1,$ & $1$ & $f$ \\
$1,$ & $2,$ & $0,$ & $1$ & $g$ \\
$1,$ & $2,$ & $1,$ & $0$ & $h$ \\
$2,$ & $2,$ & $2,$ & $2$ & $i$
\end{tabular}
\end{center}
Since there are $n$ zeros in $X, Y, Z, U$:
\begin{equation*}
  a = n, \quad a+b+c=n, \quad a+d+g=n, \quad a+e+h=n.
\end{equation*}
Thus
$$b+c=0, \quad d+g=0, \quad e+h=0.$$
Since $b, c, d,g,h$ and $e$ are non-negative integers, we have:
$$b=c=d=e=g=h=0.$$
In $X$ there are exactly $n$ twos, yielding
$$i=n.$$
Finally, there are exactly $n$ ones in $Y$, so $d+e+f=n$. But $d=e=0$, so $f=n$.

In summary: $a=f=i=n$ and $b=c=d=e=g=h=0$.
This, however, implies that $(x_i, y_i, z_i, u_i) \in \{(0,0,0,0), (1,1,1,1), (2,2,2,2)\}$ for each index $i$. Consequently, the corresponding coordinates of $X, Y, Z$, and $U$ are identical, yielding $X=Y=Z=U$. This completes the proof of the theorem. \qed

\section{Proof of Theorem \ref{thm03}.}

For the proof, let's define a special set $A_{t,n}$ using base-$8$ representations. A number $a_{3tn-2}a_{3tn-3}\dots a_2a_1a_{0\ (8)}$ belongs to $A_{t,n}$ if and only if $a_{t-1}=a_{2t-1}=a_{3t-1}=\dots=a_{(3n-1)t-1}=0$, and all other digits are $0$, $1$, or $2$. Furthermore, each of the digit sets $\{a_0, a_t, \dots, a_{(3n-1)t}\}$, $\{a_1, a_{t+1}, \dots, a_{(3n-1)t+1}\}$, $\{a_2, a_{t+2}, \dots, a_{(3n-1)t+2}\}$, \dots, $\{a_{t-2}, a_{2t-2}, \dots, a_{3nt-2}\}$ contains exactly $n$ ones, $n$ twos, and $n$ zeros.

\bigskip First, we prove the following lemma:
\begin{lemma}\lb{lem03}
If $t \ge 2$ and $n \ge 1$ are integers, then the equation $X+5Y=3U+3Z$ holds in $A_{t,n}$ if and only if $X=Y=U=Z$.
\end{lemma}

\bigskip\noindent\textbf{Proof of Lemma \ref{lem03}.} We prove the lemma by mathematical induction on $t$. For $t=2$,
\[ A_{2,n} \subseteq \{a_{6n-2}0a_{6n-4}0 \dots a_20a_0 \mid a_{2i} \in \{0,1,2\}, 0 \le i \le 3n-1\}. \]
Suppose that $X+5Y=3U+3Z$ and
\begin{align*}
    X &= x_{6n-2}0 \dots 0x_20x_{0\ (8)} \\
    Y &= y_{6n-2}0 \dots 0y_20y_{0\ (8)} \\
    U &= u_{6n-2}0 \dots 0u_20u_{0\ (8)} \\
    Z &= z_{6n-2}0 \dots 0z_20z_{0\ (8)}
\end{align*}
If we add $X+5Y$ and $3U+3Z$ digit by digit in base $8$, we might get a carry of $1$ at the even positions (or no carry at all), but no carry is ever generated at the odd positions. This is because there are zeros at the odd positions, and the previous column carries at most $1$. Thus, if we treat the numbers
\begin{align*}
    \widetilde{X} &= x_{6n-2}x_{6n-4} \dots x_2x_0 \\
    \widetilde{Y} &= y_{6n-2}y_{6n-4} \dots y_2y_0 \\
    \widetilde{U} &= u_{6n-2}u_{6n-4} \dots u_2u_0 \\
    \widetilde{Z} &= z_{6n-2}z_{6n-4} \dots z_2z_0 
\end{align*}
as vectors in $\mathbb{Z}_8^{3n}$, we see that $\widetilde{X}+5\widetilde{Y} = 3\widetilde{U}+3\widetilde{Z}$ holds in $\mathbb{Z}_8^{3n}$. Moreover, since each vector $\widetilde{X}, \widetilde{Y}, \widetilde{U}, \widetilde{Z} \in \mathbb{Z}_8^{3n}$ contains exactly $n$ zeros, $n$ ones, and $n$ twos, so by the proof of Theorem \ref{thm02a}, $\widetilde{X}=\widetilde{Y}=\widetilde{U}=\widetilde{Z}$. But this implies $X=Y=U=Z$, completing the proof of the base case of the induction.

For the inductive step, assume the statement holds for $A_{2,n}, A_{3,n}, \dots, A_{t-1,n}$. We will show it for $A_{t,n}$. Then
\[ A_{t,n} \subseteq \{a_{3tn-2}a_{3tn-3} \dots a_{(3n-1)t}0 \dots 0a_{t-2}a_{t-3}\dots a_1a_0 \mid a_i \in \{0,1,2\}\}. \]
Assume that $X+5Y=3U+3Z$, and
\begin{align*}
    X &= x_{3tn-2}x_{3tn-3} \dots x_{(3n-1)t}0 \dots 0x_{t-2} \dots x_1x_0 \\
    Y &= y_{3tn-2}y_{3tn-3} \dots y_{(3n-1)t}0 \dots 0y_{t-2} \dots y_1y_0 \\
    U &= u_{3tn-2}u_{3tn-3} \dots u_{(3n-1)t}0 \dots 0u_{t-2} \dots u_1u_0 \\
    Z &= z_{3tn-2}z_{3tn-3} \dots z_{(3n-1)t}0 \dots 0z_{t-2} \dots z_1z_0
\end{align*}
Again, we consider the sums $X+5Y$ and $3U+3Z$, but instead of examining whether a carry occurs at every single digit during the addition, we only check for carries at the $(t-1)$-th, $(2t-1)$-th, $\dots$, $((3n-1)t-1)$-th columns. These columns consist entirely of zeros. Since the incoming carry is at most $5$, the sum here never exceeds $5$. Thus, no new carry is generated, meaning the digits in the following columns must match modulo $8$:
\begin{align}
    x_0+5y_0 &\equiv 3u_0+3z_0 \pmod 8 \nonumber \\
    x_t+5y_t &\equiv 3u_t+3z_t \pmod 8 \nonumber \\
    x_{2t}+5y_{2t} &\equiv 3u_{2t}+3z_{2t} \pmod 8 \nonumber \\
    &\vdots \nonumber \\
    x_{(3n-1)t}+5y_{(3n-1)t} &\equiv 3u_{(3n-1)t}+3z_{(3n-1)t} \pmod 8 \label{eq:congruences}
\end{align}
Let us define the vectors $\widetilde{X}, \widetilde{Y}, \widetilde{U}, \widetilde{Z}$ as vectors in $\mathbb{Z}_8^{3n}$ as follows:
\begin{align*}
    \widetilde{X} &= (x_{(3n-1)t}, \dots, x_{2t}, x_t, x_0) \\
    \widetilde{Y} &= (y_{(3n-1)t}, \dots, y_{2t}, y_t, y_0) \\
    \widetilde{U} &= (u_{(3n-1)t}, \dots, u_{2t}, u_t, u_0) \\
    \widetilde{Z} &= (z_{(3n-1)t}, \dots, z_{2t}, z_t, z_0)
\end{align*}
From these congruences, the equation $\widetilde{X}+5\widetilde{Y} = 3\widetilde{U}+3\widetilde{Z}$ holds in $\mathbb{Z}_8^{3n}$. Since each of these vectors has exactly $n$ zeros, $n$ ones, and $n$ twos, Theorem \ref{thm02a} gives us $\widetilde{X}=\widetilde{Y}=\widetilde{U}=\widetilde{Z}$.
That is,
\begin{align}
  x_0&=y_0=u_0=z_0,\notag\\
  x_t&=y_t=u_t=z_t,\notag\\ 
\vdots &\phantom{aaa}\vdots\notag\\
  x_{(3n-1)t}&=y_{(3n-1)t}=u_{(3n-1)t}=z_{(3n-1)t}. \label{eq:equality_stars}
\end{align}

Now look at the original vectors $X, Y, Z, U$. By \eqref{eq:equality_stars}, if we drop the digits at positions $0,t,2t,\dots,(3n-1)t$, for
the resulting base $8$ numbers
\begin{align*}
    \widehat{X} &= x_{3tn-2}x_{3tn-3} \dots x_{(3n-1)t+1}0 \dots 0x_{t-2} \dots x_1 \\
    \widehat{Y} &= y_{3tn-2}y_{3tn-3} \dots y_{(3n-1)t+1}0 \dots 0y_{t-2} \dots y_1 \\
    \widehat{U} &= u_{3tn-2}u_{3tn-3} \dots u_{(3n-1)t+1}0 \dots 0u_{t-2} \dots u_1 \\
    \widehat{Z} &= z_{3tn-2}z_{3tn-3} \dots z_{(3n-1)t+1}0 \dots 0z_{t-2} \dots z_1
\end{align*}
we have $\widehat{X}+5\widehat{Y}=3\widehat{U}+3\widehat{Z}$. Since these truncated numbers belong to $A_{t-1,n}$, our induction hypothesis guarantees $\widehat{X}=\widehat{Y}=\widehat{U}=\widehat{Z}$.
Combining this with \eqref{eq:equality_stars} gives $X=Y=U=Z$, which completes the proof of the lemma. \qed

Next, pick an integer $t$ large enough so that $8^{t\varepsilon} > 3$, and define a positive integer $n$ satisfying
\[
  8^{3tn} \le m < 8^{3t(n+1)}.
\]
As $m \to \infty$, $n$ also grows. Thus, for a sufficiently large $m > m_0(\varepsilon)$, we have
\[ \frac{0.2^{t-1} \cdot 8^{3t\varepsilon}}{3^{3t}} \cdot \frac{\left(\frac{8^{3t\varepsilon}}{27}\right)^n}{n^{t-1}} > 1, \]
because our choice of $t$ makes $\frac{8^{3t\varepsilon}}{27} > 1$. This implies:
\begin{align*}
    \frac{0.2^{t-1} \cdot 8^{3t(n+1)\varepsilon}}{n^{t-1} \cdot 3^{3n} \cdot 3^{3t}} &> 1 \\
    0.2^{t-1} \frac{3^{3nt-3n}}{n^{t-1}} &> \frac{3^{3tn+3t}}{8^{3t(n+1)\varepsilon}} \\
    \left(0.2 \frac{3^{3n}}{n}\right)^{t-1} &> \left(3 \cdot 8^{-\varepsilon}\right)^{3t(n+1)} \\
    \left(0.2 \frac{3^{3n}}{n}\right)^{t-1} &> \left(8^{\log 3/ \log 8 - \varepsilon}\right)^{3t(n+1)} \\
    \left(0.2 \frac{3^{3n}}{n}\right)^{t-1} &> \left(8^{3t(n+1)}\right)^{\log 3/ \log 8 - \varepsilon}
\end{align*}
Since $m < 8^{3t(n+1)}$, we obtain
\begin{equation}
    \left(0.2 \frac{3^{3n}}{n}\right)^{t-1} > m^{\log 3/ \log 8 - \varepsilon}. \label{eq:lower_bound}
\end{equation}

Looking at $A_{t,n}$, we see:
\[ A_{t,n} \subseteq \{0, 1, 2, \dots, 8^{3tn-1} - 1\} \subseteq \left\{0, 1, 2, \dots, \left\lfloor \frac{m}{6} \right\rfloor - 1\right\}. \]
Thus, if $X, Y, U, Z \in A_{t,n}$ satisfy $X+5Y \equiv 3U+3Z \pmod m$, the bounds $0 \le X+5Y < m$ and $0 \le 3U+3Z < m$ force the exact equality:
\[ X+5Y = 3U+3Z. \]
By Lemma \ref{lem03}, this equation holds in $A_{t,n}$ if and only if $X=Y=U=Z$. 

Meanwhile, the size of our set is:
\[ \ab{A_{t,n}} = \left( \frac{(3n)!}{(n!)^3} \right)^{t-1}. \]
Using the bound \eqref{r401} from Theorem \ref{thm02a}, we get:
\[ \ab{A_{t,n}} > \left( 0.2 \frac{3^{3n}}{n} \right)^{t-1}. \]
Applying \eqref{eq:lower_bound} then gives:
\[ \ab{A_{t,n}} > m^{\log 3/ \log 8 - \varepsilon}, \]
which completes the proof of the theorem. \qed

\section{Sketch of the proof of Theorem \ref{thm04}.}
The proof closely follows Theorem \ref{thm03}. The main difference is tweaking the set $A_{t,n}$ to handle the larger carries generated from the coefficients $a, b, c$, and $d$.

Let $r = \lceil \log(a+b)/\log 8 \rceil$. We define $A_{t,n}$ as the set of integers with base-$8$ representations, $a_{3tn-1}a_{3tn-2}\dots a_2 a_1 a_{0 (8)}$, that satisfy:
\begin{itemize}
    \item[-] For each $i \in \{1, 2, \dots, 3n-1\}$, we have a block of $r$ consecutive zeros separating the active digits:
    \[ a_{it-1} = a_{it-2} = \dots = a_{it-r} = 0. \]
    \item[-] All other digits belong to the set $\{0, 1, 2\}$.
    \item[-] For each valid offset $k \notin \{t-1, t-2, \dots, t-r\}$, the corresponding set of digits $\{a_k, a_{t+k}, \dots, a_{(3n-1)t+k}\}$ contains exactly $n$ ones, $n$ twos, and $n$ zeros.
\end{itemize}

When evaluating $aX + bY = cZ + dU$, the sum in a single column can be at most $2a + 2b + C$, where $C$ is the carry from the previous position (since the maximum digit in our set is $2$). The maximum possible carry $C$ must satisfy $C \le \lfloor(2a + 2b + C)/8\rfloor$, which implies $7C \le 2(a+b)$, and thus $C \le \frac{2}{7}(a+b) < a+b$. 

The block of $r$ zeros at positions $it-1, it-2, \dots, it-r$ is specifically chosen because $8^r \ge a+b > C$. At these specific positions, the sum simply evaluates to $0 \cdot a + 0 \cdot b + C = C$. Since the incoming carry $C$ is always strictly less than $8^r$, it can be fully represented in base $8$ using at most $r$ digits. This guarantees that any carry generated by the active digits is completely absorbed within these separating zero digits and does not spill over to the next active block at position $it$.

Furthermore, since $a \equiv 1 \pmod 8$, $b \equiv 5 \pmod 8$, and $c \equiv d \equiv 3 \pmod 8$, the modulo $8$ congruences at the active digit positions simplify right back to $X + 5Y \equiv 3U + 3Z \pmod 8$. Thus, we can reuse the digit-by-digit logic from Lemma \ref{lem03} exactly as before.

The rest of the induction and density estimates follow the exact same steps as in Theorem \ref{thm03}, yielding the same asymptotic lower bound. \qed

\bigskip\noindent\textbf{Declaration on the use of artificial intelligence.}
The computer searches and the generation of the tables presented in Section 4 were performed using the SageMath mathematical software system, with the code executed in the CoCalc cloud-based environment. The authors utilized the Google Gemini large language model (LLM) to assist in drafting certain parts of the manuscript and in writing the SageMath code. %Following the use of this tool, the authors reviewed and edited the content as necessary and take full responsibility for the integrity of the final text and code.

\noindent
Katalin Gyarmati\\
E\"otv\"os Lor\'and University, Institute of Mathematics,\\
H-1117 Budapest, P\'azm\'any P\'eter st. 1/C, Hungary\\
\noindent E-mail address: \texttt{katalin.gyarmati@ttk.elte.hu}

\end{document}